\documentclass{article}

\usepackage{amsmath,amssymb}%
\usepackage{upgreek}
\usepackage{mathrsfs,txfonts,yfonts}
\usepackage{graphicx}
\usepackage[all]{xy}
\usepackage{ntheorem}
\numberwithin{equation}{section}%
\usepackage[colorlinks, linkcolor=blue, anchorcolor=blue, citecolor=blue]{hyperref}
\usepackage{esint}

\theorembodyfont{\upshape}
\newtheorem{thm}{Theorem}[section]
\newtheorem{cor}{Corollary}[section]
\newtheorem{lemma}{Lemma}[section]

\newtheorem*{proof}{Proof}

\newcommand*\lap{\mathop{}\!\mathbin\bigtriangleup}

\def\Ka{\mathrm{K\ddot{a}hler}}
\def\i{\sqrt{-1}}

\def\dd{\partial\bar{\partial}}
\def\tr{\mathrm{tr}}
\def\Am{\mathrm{Amp\Grave{e}re}}
\def\Linfty{L^{\infty}}

\begin{document}

\begin{center}
{\Large The $\Linfty$ estimate for parabolic complex Monge-Amp\`ere equations}

\medskip
\centerline{Qizhi Zhao}

\begin{abstract}
{\footnotesize } Following the recent development in \cite{CC3,GPT}, we derived the $\Linfty$ estimate for $\Ka$-Ricci flows under a weaker assumption. The technique also extends to more general cases coming from different geometric backgrounds.
\end{abstract}

\end{center}

\baselineskip=15pt
\setcounter{equation}{0}
\setcounter{footnote}{0}

\section{Introduction}

In this paper, we will derive the $\Linfty$ estimate for the $\Ka$-Ricci flow,
\begin{equation}\label{MainFlow}
    \left\{
    \begin{aligned}
        &\partial_t\varphi=\log\left(\frac{\omega_{\varphi}{}^n}{e^{nF}\omega_0{}^n}  \right)\\
        &\varphi(\cdot,0)=\varphi_0(\cdot),
    \end{aligned}
    \right.
\end{equation}
under the assumption that the $p$-entropy $\mathrm{Ent}_p(F)=\int_M |F|^pe^{nF}\omega_{0}{}^n$ is bounded and moreover $\int_M F\omega_0{}^n$ has a lower bound. Here is our main theorem,

\begin{thm}\label{MainThm}
    Let us consider the flow equation (\ref{MainFlow}) on $M\times[0,T)$, where $M$ is an $n$-dimensional compact $\Ka$ manifold. Let $F$ be a space function, i.e. $F:M\rightarrow \mathbb{R}$. Assume for some $p>n+1$, the $p$-entropy $\mathrm{Ent}_p(F)=\int_{M}|F|^pe^{nF}\omega_{0}{}^n$ is bounded and $\int_M nF\omega_0{}^n \geq - K$. Moreover, suppose $\varphi$ be a $C^2$ solution and let $\tilde{\varphi}=\varphi-\fint\varphi\omega_0{}^n$ be a normalization which has the zero integral. Then we have the $\Linfty$ estimate, $$\|\Tilde{\varphi}\|_{\Linfty\big(M\times[0,T)\big)}\leq C,$$ where $C$ depends on $n,\  \omega_0,\ \varphi_0,\ p,\ K$ and $\mathrm{Ent}_p(F)$.
\end{thm}

In \cite{Yau1}, Yau applied the method of Moser iteration to derive the $L^{\infty}$ estimate for Monge-$\Am$ equation when $\|F\|_{L^p}$ is bounded for some $p>n$. Later, Kołodziej \cite{Kolo} gave another proof by using the pluripotential theory under a weaker assumption that $\|F\|_{L^p}$ is bounded for some $p>1$. More recently, Guo-Phong-Tong \cite{GPT} recovered Kołodziej's estimate by a PDE method which was partly motivated by Chen-Cheng's breakthrough on the cscK metric \cite{CC1}.

The $\Ka$-Ricci flow was firstly studied by Cao. In \cite{Cao1}, he gave an alternative proof of Calabi's conjecture for $c_1(M)=0$ and $c_1(M)<0$ which investigated the estimates for the $\Ka$-Ricci flow instead of Monge-$\Am$ equations. There are abundant results on $\Ka$-Ricci flows, see \cite{EGZ1,EGZ2, GLZ, JianShi23} for example. Our result requires a weaker regularity on the right hand side than Cao's $\Linfty$ estimate and can be viewed as a parabolic analogue of \cite{GPT, WWZ}.

There are some technical improvements in our paper compared with some previous results in \cite{CC3,GP2,GP1,GPT}. The difficulty for the flow problem is that we want to derive an $\Linfty$ estimate independent on $T$. But the original auxiliary equation may not serve as a good choice. Our approach is to consider a local version of auxiliary flows instead, see section \ref{sectaux}. Compared with elliptic version of $L^{\infty}$ estimate, our theorem requires an extra integral condition. Rewrite the equation (\ref{MainFlow}) by $\omega_{\varphi}^n = e^{\dot{\varphi} + nF} \omega_0^n$, we could see that the $L^{\infty}$ estimate of $\varphi$ comes from some $p$-Entropy bounds on $e^{\dot{\varphi} + nF}$. Roughly speaking, we need not only the $p$-Entropy bound controlls on $e^{nF}$ but also some upper bounds on $\dot{\varphi}$. In $\Ka$-Ricci flow, Cao proved the supremum of $\dot{\varphi}$ can be controlled by the infimum of $e^{nF}$ when $F$ is smooth. Indeed such estimates can be generalized to general parabolic Monge-Amp\`ere flows. However, in our theorem, we can improve the point-wise condition by some integral condition on $F$.

%It is well known that the supremum of $\dot{\varphi}$ depends on the infimum of $e^{nF}$, but in this paper the point-wise infimum condition of $e^{nF}$ could be improved by an integral condition.

There are two directions to generalize our main theorem \ref{MainThm}. As in \cite{CC3}, we can replace the Monge-$\Am$ operator on the right hand side by a more general nonlinear operator $\mathcal{F}$. Let $h_{\varphi}$ be an endomorphsim defined locally by $(h_{\varphi})^j_k=g^{j\bar{m}}(\omega_{\varphi})_{\bar{m}k}$, where $g^{j\bar{m}}$ is a local coordinate of the initial metric $\omega_{0}=\i g_{j\bar{m}}\mathrm{d}z^j\wedge\mathrm{d}\bar{z}^{m}$. Let $\lambda[h_{\varphi}]$ be the vector of eigenvalues of $h_{\varphi}$, and consider the nonlinear operator $\mathcal{F}:\Gamma\subset\mathbb{R}^n\rightarrow\mathbb{R}_{+}$ with the following four conditions,

\begin{enumerate}
    \item The domain $\Gamma$ is a symmetric cone with $\Gamma_n\subset\Gamma\subset\Gamma_1$, where $\Gamma_k$ is defined to be the cone of vectors $\lambda$ with $\sigma_j(\lambda)>0$ for $1\leq j\leq k$, where $\sigma_j$ is the $j$-th symmetric polynomial in $\lambda$;
    \item $\mathcal{F}(\lambda)$ is symmetric in $\lambda\in\Gamma$ and it is of homogeneous degree $r$;
    \item $\frac{\partial \mathcal{F}}{\partial \lambda_j}>0$ for each $j=1,\dots,n$ and $\lambda\in\Gamma$;
    \item There is a $\gamma>0$ such that
    \begin{equation}
        \prod_{j=1}^n\frac{\partial \mathcal{F}(\lambda)}{\partial\lambda_j}\geq\gamma\mathcal{F}^{n(1-\frac{1}{r})}, \forall\lambda\in\Gamma.
    \end{equation}
\end{enumerate}

The above requirements come from \cite{GPT}, and there is a slight modification on the last condition since the homogeneous degree of $\mathcal{F}$ is $r$ under our assumption. The complex Hessian operators and $p$-Monge-$\Am$ operators are examples. More examples can be found in \cite{HL}. Here is our first generalization,

\begin{thm}\label{genmainthm}
    Let $\varphi$ be a $C^2$ solution of the following flow 
    \begin{equation}\label{genflow}
        \left\{
        \begin{aligned}          
            &\partial_t\varphi=\log\left(\frac{\mathcal{F}(\lambda[h_{\varphi}])}{e^{rF}} \right)\\
            &\varphi(\cdot,0)=\varphi_0
        \end{aligned}
        \right.
    \end{equation}
    on $M\times[0,T)$, where $M$ is an $n$-dimensional compact $\Ka$ manifold. Let $F$ be a space function, i.e. $F:M\rightarrow\mathrm{R}$. Assume for some $p>n+1$, the $p$-entropy $\mathrm{Ent}_p(F)=\int_{M}|F|^pe^{nF}\omega_{0}{}^n$ is bounded and $\int_M F\omega_0{}^n \geq - K$. Then we have the $L^{\infty}$ estimate on the normalization $\tilde{\varphi}$, $$\|\tilde{\varphi}\|_{\Linfty\big(M\times[0,T)\big)}\leq C,$$ where $C$ depends on $n,\  \omega_0,\ \varphi_0,\ p,\ K,\ \gamma$ and $\mathrm{Ent}_p(F)$.
\end{thm}

% Note that in the flow (\ref{genflow}), the coefficient in the right hand side is $e^{rF}$ where $r$ is the homogeneous degree of $\mathcal{F}$.

Another direction of the generalization is motivated by Chen-Cheng. The $\Linfty$ estimate for the inverse Monge-$\Am$ flow
\begin{equation}\label{Inverse}
    \left\{
    \begin{aligned}
        &(-\partial_t u)\omega_{\varphi}^n=e^{nF}\omega_{0}{}^n\\
        &\varphi(\cdot,0)=\varphi_0.
    \end{aligned}
    \right.
\end{equation}
is considered in \cite{CC3}. Indeed, we can consider the following general complex Monge-Amp\`ere flow 
\begin{equation}\label{star}
    \left\{
    \begin{aligned}
        &\partial_t\varphi=\Theta\left(\frac{\omega_{\varphi}{}^n}{e^{nF}\omega_{0}{}^n}\right),\\
        &\varphi(\cdot,0)=\varphi_0,
    \end{aligned}
    \right.
\end{equation}
where $\Theta: \mathbb{R}_{+} \to \mathbb{R}$ is a strictly increasing smooth function. In \cite{PZ1}, Picard-Zhang proved the long time existence and convergence of the flow (\ref{star}) under the assumption that $F\in C^\infty(M, \mathbb{R})$. It is the $\Ka$-Ricci flow when $\Theta(y)=\log y$, and the inverse Monge-Amp\`ere flow (\ref{Inverse}) when $\Theta(y)=-\frac{1}{y}$. Indeed, the general parabolic Monge-Amp\`ere flow (\ref{star}) also arises from many other geometric problems. For example, when $\Theta(y)=y$, this is the flow reduced from the Anomaly flow with conformal K\"ahler initial data (see \cite{PPZ}), when $\Theta(y)=y^\frac{1}{3}$, this is the reduction of $G_2$-Laplacian flow over 7 dimensional manifold (\cite{PS}). We can apply the new technique to prove the $L^\infty$ estimate of the solution $\varphi$ to equation (\ref{star}) under a weaker assumption on $F$.

\begin{thm}\label{genmainthm2}
Assume $\Theta(y) = -\frac{1}{y}$, $y$ or $y^{\frac{1}{3}}$, $\mathrm{Ent}_p(F)$ is bounded. Moreover, consider constant $K$ to be $K=\max\left(0,\int_M\Theta(e^{-nF})\omega_0{}^n\right)$. Then, there exists a constant $C$ depending on $n$, $\omega_{0}$, $\varphi_0$, $p$, $\max(K, 0)$ and $\mathrm{Ent}_p(F)$, such that $$\|\tilde{\varphi}\|_{\Linfty\big(M\times[0,T)\big)}\leq C,$$ where $\tilde{\varphi}$ is a normalization of a $C^2$ solution of the flow (\ref{star}).
\end{thm}

Since $\Theta<0$ in the inverse Monge-Amp\`ere equation, we have $K\equiv0$, which means there is no extra condition for this case.

In the rest of this paper, a constant is called universal if it depends only on $n$, $\omega_0$, $\varphi_0$, $p$, $\gamma$, $K$ and $\mathrm{Ent}_p(F)$. 

\

{\bf Acknowledgement.} The author would like to thank Prof. Xiangwen Zhang for the helpful discussions and numerous encouragement throughout the project. The author is also grateful for Prof. Duong H. Phong and Prof. Bin Guo's interest and useful comments. And thanks Prof. Freid Tong for some private conversations which are enlightening.

\section{Auxiliary equations}\label{sectaux}

In this section, we want to find suitable auxiliary equations as in \cite{GPT} and \cite{CC3}. To motivate what a good auxiliary equation should be, we first consider the parabolic Monge-$\Am$ flow  (\ref{star}). To have some monotonicity properties of the auxiliary solutions, we prefer a flow with negative time derivative of $\psi$. Thus the inverse Monge-$\Am$ flow, see (\ref{Inverse}), is a good candidate.

Let us consider the inverse Monge-$\Am$ flow 

\begin{equation}\label{auxeq1}
    \left\{
    \begin{aligned}
        &\left(-\dot{\psi_s}\right)\omega_{\psi}{}^n=f_se^{nF}\omega_0{}^n,\\
        &\psi_s(\cdot,0)=0,
    \end{aligned}
    \right.
\end{equation}
where $f_s=\frac{(-\varphi-s)_{+}}{A_s}$, $A_s=\int_{\Omega_s}(-\varphi-s)e^{nF}\omega_0{}^n\mathrm{d}t$ and $\Omega_s=\{(z,t)|-\varphi(z,t)-s>0\}$.

But such flow has singularities, since the factor $(-\varphi-s)_{+}/A_s$ is not smooth in a neighborhood of $\partial\Omega_s$. Thus we need to consider a sequence of smooth functions $\tau_k$ which convergences uniformly to $\chi_{\mathbb{R}^{+}}$, and replace the right hand side $f_s$ by $\frac{\tau_k( -\varphi-s)}{\int_{\Omega}\tau_k(-\varphi-s)e^{nF}}$. By the Lebesgue dominant theorem, $\psi_{s,k}$ converges to $\psi_s$ uniformly, which means we can always take a limit in the inequalities to get the desired estimates as in \cite{GPT} and \cite{CC3}. To simplify our computations, we will keep using (\ref{auxeq1}) as our auxiliary equation. % and there will be no essential differences when we choose $\tau_k$ and then take a limit.

Another crucial modification of our auxiliary flow is we should restrict the integration over the time slices. To make our idea more clear, we need Lemma 2.1 as well as Corollary 2.2 in \cite{CC3} which will be stated below. For reader's convenience, I will also put the proof in \cite{CC3} here.

\begin{lemma}\label{cclemma}
    Consider the inverse Monge-$\Am$ flow
    \begin{equation*}
        \left\{
        \begin{aligned}
            &\left(-\dot{\varphi}\right)\omega_{\varphi}{}^n=e^{nF}\omega_0{}^n\\
            &\varphi|_{t=0}=\varphi_0.
        \end{aligned}
        \right.
    \end{equation*}
Assume that $\int_{M\times[0,T]}e^{nF}\omega_{0}{}^n\mathrm{d}t=C_1<\infty$, then we have $\sup_M|\varphi|\leq C,$ where C depends on $n, \omega_{0}, C_1$ and $\|\varphi_0 \|_{L^{\infty}}$. 
\end{lemma}

\begin{proof}
    Since $\dot{\varphi}<0$, we can get the upper bound by $$\sup_M\varphi\leq\sup_M\varphi_0\leq\|\varphi_0 \|_{L^{\infty}}.$$

    To estimate the lower bound of $\sup_M\varphi$, let us consider the $I$-functional $$I(\varphi)=\frac{1}{n+1}\int_M\varphi\sum_{j=0}^n\omega_{0}{}^{n-j}\wedge\omega_{\varphi}{}^j,$$ and its derivative $$\frac{\mathrm{d}}{\mathrm{d}t}I(\varphi)=\int_M\partial_t\varphi\omega_{\varphi}{}^n=-\int_M e^{nF}\omega_{0}{}^n.$$

    Therefore, for any $t'\in[0,T]$, we have $$I(\varphi)-I(\varphi_0)=\int^{t'}_0\frac{\mathrm{d}}{\mathrm{d}t}I(\varphi)\mathrm{d}t=-\int_{M\times[0,t']} e^{nF}\omega_{0}^n\geq -\int_{M\times[0,T]} e^{nF}\omega_{0}^n=C_1,$$
which implies $I(\varphi)$ is bounded from below on $[0, T]$.

    The lower bound estimate of $\int_M \varphi\omega_{0}{}^n$ comes from the integration by part,

    \begin{align*}
        \int_M\varphi\omega_{0}{}^n-I(\varphi)=&\int_M \varphi\frac{1}{n+1}\sum^n_{j=0}\omega_{0}{}^{n-j}\wedge(\omega_{0}{}^j-\omega_{\varphi}{}^j)\\
        =&\frac{1}{n+1}\int_M\varphi\sum^n_{j=0}\omega_{0}{}^{n-j}\wedge\i\dd(-\varphi)\sum^{j-1}_{l=0}\omega_{0}{}^{j-1-l}\wedge\omega_{\varphi}{}^l\\
        =&\frac{1}{n+1}\int_M \i\partial\varphi\wedge\bar{\partial}\varphi\wedge\sum^n_{j=0}\sum^{j-1}_{l=0}\omega_{0}{}^{n-1-l}\wedge\omega_{\varphi}{}^l\geq 0.
    \end{align*}
Therefore, we have $\int_M\varphi\omega_0{}^n\geq C_1$. The last step is to estimate $\sup_M\varphi$ by using a property of $\omega_{0}$-psh functions. Suppose $\varphi$ is a $\omega
_0$-psh function, then we have $$\left|\frac{1}{\mathrm{Vol(M, \omega_0)}}\int_M\varphi\omega_{0}{}^n-\sup_M\varphi\right|\leq C.$$

    Thus we complete the prove, since the last constant depends only on the background metric $\omega_0$.

\end{proof}

\begin{cor}\label{alpha}
    There exists a constant $\alpha_0>0$ depending only on $\omega_{0}$, such that $$\sup_{t\in[0,T]}\int_M e^{-\alpha \varphi}\omega_{0}{}^n\leq C_2,$$ where $C_2$ depends on  $M, \omega_{0}, C_1$ and $\|\varphi_0 \|_{L^{\infty}}$ as well.
\end{cor}

This is a flow version of H$\mathrm{ \ddot{o}}$rmander's result, see lemma 4.4 in \cite{Horm1} and \cite{Tian1} for local and global version of such integral estimate respectively.

\begin{proof}
    Since $\varphi$ is a $\omega_{0}$-psh function for every $t\in[0.T]$, we have $$\sup_{t\in[0,T]}\int_M e^{-\alpha (\varphi-\sup_M\varphi)}\omega_{0}{}^n\leq C.$$

    From Lemma \ref{cclemma}, the uniform bound of $\sup_M\varphi$ gives us the desired inequality.
\end{proof}

The above corollary \ref{alpha} gives us a uniform bound on each time slice. In the following sections, the corollary will be applied to the integrals over the space-time. If we consider original auxiliary equation (\ref{auxeq1}), it will be hard to get a time-independent $\Linfty$ estimate. Thanks for the uniform bound on each time slice, we can restrict on some smaller time intervals to construct our auxiliary equations.

To get the $L^{\infty}$ estimate, we need also consider the normalization in theorem \ref{MainThm}, which follows the same normalization in \cite{PZ1}. Moreover we need to use the domain $\Tilde{\Omega}_s=\{(z,t)|-\tilde{\varphi}(z,t)-s>0\}$ to replace $\Omega_s$.

For any $t_0\in[0,T-1)$, let us consider a family of regions $\Omega_{s,t_0}=\tilde{\Omega}_{s}\cap\left(M\times[t_0,t_0+1]\right)$, and choose a family of auxiliary equations,

\begin{equation}\label{aux}
    \left\{
    \begin{aligned}
        &\left(-\dot{\psi}_{s,t_0}\right)\omega_{\psi}{}^n=f_{s,t_0}e^{nF}\omega_0{}^n\\
        &\psi_{s,t_0}(\cdot,0)=0,
    \end{aligned}
    \right.
\end{equation}
where $A_{s,t_0}=\int_{\Omega_{s,t_0}}\chi_{_{\Omega_{s,t_0}}}(-\tilde{\varphi}-s)\omega_{0}{}^n\mathrm{d}t$ and $f_{s,t_0}$ is the normalization of $\chi_{_{\Omega_{s,t_0}}}(-\tilde{\varphi}-s)$.

By corollary \ref{alpha}, the cut-off along the time controls the integration inequality over the space-time
\begin{equation}\label{alphainv}
    \int_{\Omega_{s,t_0}}e^{-\alpha \psi_{s,t_0}}\omega_{0}{}^n\mathrm{d}t\leq C_2,
\end{equation}
which will be crucial to get the desired $L^{\infty}$ estimate since the above inequality (\ref{alphainv}) is $T$-independent.

%The local version of auxiliary equation (\ref{aux}) could help us get the time-independent $\Linfty$ estimate.

The family of auxiliary equations (\ref{aux}) meets the same problem as in (\ref{auxeq1}). To be precise, we also need to apply $\tau_k$ to remove the singularities. For the same consideration, we will still keep using (\ref{aux}) in the following sections.

The extra integral condition was chosen to make the normalization $\tilde{\varphi}$ satisfy the following three properties.

\begin{lemma}\label{gap}
    Let $\tilde{\varphi}$ be given in theorem (\ref{MainThm}), then we have,
    \begin{enumerate}
        \item $\sup_{t\in[0,T)}\int_M\dot{\varphi}\omega_0{}^n\leq C_3$,
        \item $\tilde{\varphi}\leq C_3$,
        \item $\int_M|\tilde{\varphi}|\omega_0{}^n\leq C_3$,
    \end{enumerate}
    where $C$ is universal.
\end{lemma}

\begin{proof}
    For 1, it comes directly from the following estimates,
    \begin{align*}
        \fint_M\dot{\varphi}\omega_0{}^n&=\fint_M\log\left(\frac{\omega_{\varphi}{}^n}{e^{nF}\omega_0{}^n}\right)\omega_0{}^n=\fint\log\left(\frac{\omega_{\varphi}{}^n}{\omega_0{}^n}\right)\omega_0{}^n - \fint\log\left(e^{nF}\right)\omega_0{}^n\\
        &\leq\log\left(\fint\omega_{\varphi}{}^n\right)- \fint\log\left(e^{nF}\right)\omega_0{}^n = -\fint\log\left(e^{nF}\right)\omega_0{}^n\\
        &\leq K.
    \end{align*}
    The first estimate comes from Jensen's inequality while the second one comes from the assumption on $F$. The average integral is chosen with respect to $V=\int_M\omega_0{}^n$.

    To prove 2 and 3, let's start from Green's formula,
    $$\Tilde{\varphi} = \fint_M\Tilde{\varphi}\omega_0{}^n - \int_MG\lap\tilde{\varphi}\omega_0{}^n=- \int_MG\lap\tilde{\varphi}\omega_0{}^n=- \int_MG\lap\varphi\omega_0{}^n,$$ where $G$ is the Green's function with respect to $\omega_0$.

    It is well known that the Green's function $G$ could be shifted to be non-negative and with $L^1$ norm bound. Combine with the inequality $\tr_{\omega_0}\omega_{\varphi}=n+\lap\varphi> 0$ and Green's formula, we have the following universal estimate $\tilde{\varphi}\leq C_3$.

    Let $I_{+}$ and $I_{-}$ be the integral of positive and negative parts of $\tilde{\varphi}$ respectively. Then we have $0=\int_M\Tilde{\varphi} = I_+ - I_-$ and $I_+\leq C_3V$. Thus $$\int_M|\tilde{\varphi}| = I_++I_-=2I_+\leq C_3.$$

    %and combine with Green's formula, we have

\end{proof}

%The following three properties shothe technical reason on choosing this normalization

% We will use the following well-known lemma to end this section.

% \begin{lemma}
    % Let $\varphi$ be a solution of (\ref{MainFlow}), then we have $\dot{\varphi}=\lap{\dot{\varphi}}$, moreover $$\sup_{M\times[0,T)}\dot{\varphi}=\sup_M\dot{\varphi}(x,0).$$
% \end{lemma}

% The following corollary helps us estimate the integral $\fint_M\dot{\varphi}\omega^n_0$,

% \begin{cor}
    % Let $\varphi$ be as above, then the integral $\fint_M\dot{\varphi}(x,t)\omega^n_0$ is bounded by a uniform constant $C_3$ for all $t\in[0,T)$.
% \end{cor}

% \begin{proof}   
    % It is sufficient to show that $\int_M\dot{\varphi}(x,0)\omega^n_0$ is uniformly bounded. From the flow equation (\ref{MainFlow}), we have
    % \begin{align*}
        % \int_M\dot{\varphi}(x,0)\omega^n_0&=\int_M\log\left(\frac{\omega_{\varphi}^n}{e^{nF}\omega_0^n}  \right)(x,0)\omega^n_0\\
        % &=\int_M\log\left(\frac{\omega_{\varphi}^n}{\omega_0^n}  \right)(x,0)\omega_0^n-n\int_M F(x)\omega^n_0.\\
    % \end{align*}
% \end{proof}

% Since the first term is bounded by initial value and the second term is bounded by the entropy, we complete the proof of this corollary.

%We, of course, should use the previous sequence of functions $\tau_k$ to avoid the singularities of our auxiliary flows, and the similar Lebesgue theorem argument follows. But to simplify our notations, we will use equation \ref{aux} as the family of auxiliary equations and get the following computations formally.

\section{Entropy bounded by energy}\label{secent}

From section \ref{sectaux}, we get a good choice of a family of auxiliary equations (\ref{aux}). The following lemma is crucial for the proof of theorem \ref{MainThm},

\begin{lemma}\label{keylemma}
    Suppose $\varphi$ is as in theorem \ref{MainThm}, and $\psi_{s,t_0}$ is a solution of the auxiliary flow (\ref{aux}), then there are constants $\beta,\ \epsilon$ and $\Lambda$ with $\beta=(n+1)/(n+2)$, $\epsilon^{n+2}=\left(\frac{n+2}{n+1} \right)^{n+2}\Lambda$ and $\epsilon^{n+2}=c^{n+1}A_{s,t_0}$, where $c$ is a universal constant, such that 
    \begin{equation}\label{key}
        -\epsilon(-\psi_{s,t_0}+\Lambda)^{\beta}-\tilde{\varphi}-s\leq0,
    \end{equation}
    holds on $M\times[t_0,t_0+1]$.
\end{lemma}

%In the following of this paper, the estimates we desired should be independent with the choice of time interval $[t_0.t_0+\delta$.

%Without otherwise specifying, we will use $\Omega_s$ to denote the domain $\Omega_{s,t_0,\delta}$, and this simple notation will not cause any confusion since all estimates are independent on the choice of $(t_0,\delta)$.

%denote a family of domains with a fixed pair of $(t_0,\delta)$. It will not cause confusion, because all of the estimates we will be independent on the choice of $[t_0.t_0+\delta]$.

We will use $\psi$ to denote $\psi_{s,t_0}$ in the following estimates. Let us consider the test function $H=-\epsilon\left(-\psi+\Lambda\right)^{\beta}-\tilde{\varphi}-s$ and the linearization operator $L=-\frac{\partial}{\partial t}+\lap_{\omega_{\varphi_t}}$ of the the $\Ka$-Ricci flow (\ref{MainFlow}). The idea of the following argument comes from \cite{GPT}.

Suppose if the maximum of $H$ is attained at some point $x_0=(z_0,t_0)$ outside $\Omega^{\circ}_{s,t_0}$, then we have $H\leq H(x_0)\leq-\tilde{\varphi}(x_0)-s\leq0$. To complete the proof, we only need to assume the maximal point $x_0$ of $H$ is in $\Omega^{\circ}_{s,t_0}$, and then apply the maximal principle.

% If $x_0\in\Omega^{\circ}_{s,t_0}$, then we will apply the maximal principle at such point $x_0$ to get the estimate.

%The setting on $\beta$, $\epsilon$ and $\Lambda$ should make the following estimate holds,

Since we have chosen $0<\beta<1$ and $1-\beta\epsilon\Lambda^{\beta - 1}$, we can reduce the inequality derived from the maximal principle. To be specific, we have the following inequality holds at the maximal point $x_0$,
\begin{align*}
    0\geq& LH= -\beta\epsilon(-\psi+\Lambda)^{\beta-1}\dot{\psi}+\dot{\varphi}-{\tiny \fint}\dot{\varphi}\omega_0{}^n\\
    &+\beta\epsilon(-\psi+\Lambda)^{\beta-1}\lap_{\omega_{\varphi}}\psi+\beta(1-\beta)\epsilon(-\varphi+\Lambda)^{\beta-2}|\partial\varphi|^2_{\omega_{\varphi}}-\lap_{\omega_{\varphi}}\varphi\\
    \geq&\beta\epsilon(-\psi+\Lambda)^{\beta-1}\left(-\dot{\psi}\right)-\left(-\dot{\varphi}\right)+\beta\epsilon(-\psi+\Lambda)^{\beta-1}\lap_{\omega_{\varphi}}\psi-\lap_{\omega_{\varphi}}\varphi-C_3\\
    =&\beta\epsilon(-\psi+\Lambda)^{\beta-1}\left(-\dot{\psi}\right)-\left(-\dot{\varphi}\right)+\beta\epsilon(-\psi+\Lambda)^{\beta-1}\tr_{\omega_{\varphi}}\omega_{\psi}-\\
    &\tr_{\omega_{\varphi}}\omega_{\varphi}+\left(1-\beta\epsilon(-\psi+\Lambda)^{\beta-1} \right)\tr_{\omega_{\varphi}}\omega_0\\
    \geq&\beta\epsilon(-\psi+\Lambda)^{\beta-1}\left(-\dot{\psi}+\tr_{\omega_{\varphi}}\omega_{\psi}\right)-\left(-\dot{\varphi}+C_3+n\right),
\end{align*}
the last estimate comes from $1-\beta\epsilon(-\psi+\Lambda)^{\beta-1}\geq 1-\beta\epsilon\Lambda^{\beta-1}=0$.

%since $0<\beta<1$ and $1-\beta\epsilon(-\psi+\Lambda)^{\beta-1}\geq 1-\beta\epsilon\Lambda^{\beta-1}=0$.

%This gives the conditions on the constants should be $0<\beta<1$ and $1-\beta\epsilon\Lambda^{\beta-1}=0$. Since the first condition implies $\beta(1-\beta)>0$ and the two conditions together imply $1-\beta\epsilon(-\psi+\Lambda)^{\beta-1}\geq 1-\beta\epsilon\Lambda^{\beta-1}=0$, noticing that $x^{1-\beta}$ is decreasing and $\psi\leq0$ on $M\times[0,T]$.

Then we need to deal with the factor $-\dot{\psi}+\tr_{\omega_{\varphi}}\omega_{\psi}$ which is the main term of the estimate. By the geometric-arithmetic inequality, we have
\begin{equation}\label{geo}
    -\dot{\psi}+\tr_{\omega_{\varphi}}\omega_{\psi}\geq-\dot{\psi}+n\left(\omega_{\psi}^n/\omega_{\varphi}^n\right)^{\frac{1}{n}}.
\end{equation}

%$$-\dot{\psi}+\tr_{\omega_{\varphi}}\omega_{\psi}\geq-\dot{\psi}+n\sqrt[n]{\omega_{\psi}^n/\omega_{\varphi}^n}.$$ 

Combine (\ref{geo}) with two flow equations (\ref{MainFlow}) and (\ref{aux}), and use the geometric-arithmetic inequality again, we have 
\begin{equation}\label{4keylemma}
    \begin{aligned}
        -\dot{\psi}+\tr_{\omega_{\varphi}}\omega_{\psi}&\geq-\dot{\psi}+n\left(\omega_{\psi}^n/\omega_{\varphi}^n\right)^{\frac{1}{n}}\\
        &\geq-\dot{\psi}+nf^{1/n}\exp\left(-\frac{1}{n}\dot{\varphi}\right)(-\dot{\psi})^{-1/n}\\
        &\geq(n+1)f^{1/(n+1)}\exp\left(-\frac{1}{n+1}\dot{\varphi}\right).
    \end{aligned}
\end{equation}
% by applying geometric-arithmetic inequality again.

%$$-\dot{\psi}+n\sqrt[n]{\omega_{\psi}^n/\omega_{\varphi}^n}\geq-\dot{\psi}+nf^{1/n}\exp\left(-\frac{1}{n}\dot{\varphi}\right)(-\dot{\psi})^{-1/n}\geq(n+1)f^{1/(n+1)}\exp\left(-\frac{1}{n+1}\dot{\varphi}\right).$$

%Plug in to previous estimate, we have

Thus replace $-\dot{\psi}+\tr_{\omega_{\varphi}}\omega_{\psi}$ by (\ref{4keylemma}), we have the following estimate holds at $x_0$
\begin{align*}
    0\geq&LH\geq(n+1)\beta\epsilon(-\psi+\Lambda)^{\beta-1}f^{1/(n+1)}\exp\left(-\frac{1}{n+1}\dot{\varphi}\right)-\left(-\dot{\varphi}+C_3+n\right)\\
    \geq&\left[(n+1)\beta\epsilon(-\psi+\Lambda)^{\beta-1}f^{1/(n+1)}-\left(-\dot{\varphi}+C_3+n\right)\exp\left(\frac{1}{n+1}\dot{\varphi}\right)\right]\exp\left(-\frac{1}{n+1}\dot{\varphi}\right).
\end{align*}

%From above we see, at $x_0$, we have

Since the exponential function is positive, we can simplify it by
\begin{equation}\label{middlesteps}
    0\geq(n+1)\beta\epsilon(-\psi+\Lambda)^{\beta-1}f^{1/(n+1)}+\left(\dot{\varphi}-C_3-n\right)\exp\left(\frac{1}{n+1}\dot{\varphi}\right),
\end{equation}
which is close to the inequality (\ref{key}) in lemma \ref{keylemma}. Let us consider a function $h(x)=(x-n-C_3)\exp\left(\frac{1}{n+1}x\right)$. It is easily to see this function has a universal lower bound $-C_4$, where $C_4=(n+1)\exp(\frac{C_3-1}{n+1})$. Thus we have
\begin{equation}\label{middle}
    (n+1)\beta\epsilon(-\psi+\Lambda)^{\beta-1}f^{1/(n+1)}-C_4\leq0.
\end{equation}

%There will meet several universal constant in the following proof which will all be denoted as $c$.

%$$(n+1)\beta\epsilon(-\psi+\Lambda)^{\beta-1}f^{1/(n+1)}-c\leq0.$$

The inequality (\ref{middle}) is equivalent to
\begin{equation}
   \left(\frac{(n+1)\beta}{C_4}\right)^{n+1}\epsilon^{n+1}\frac{-\tilde{\varphi}-s}{A_{s,t_0}}\leq(-\psi+\Lambda)^{(n+1)(1-\beta)} ,
\end{equation}
which is the test function when we chose the constants $\beta$, $\epsilon$ and $\Gamma$ as the statement of lemma \ref{keylemma}. Thus we have $$H\leq H(x_0)\leq 0,$$ which completes the proof.

%$c_4\epsilon^{n+1}\frac{-\varphi-s}{A_s}\leq(-\psi+\Lambda)^{(n+1)(1-\beta)}$. Comparing with the test function $H$ gives us the setting on the constants.

From the above lemma \ref{keylemma}, we have 
\begin{equation}\label{originalineq}
    \frac{-\tilde{\varphi}-s}{A_{s,t_0}^{1/(n+2)}}\leq \Bigg(c(-\psi+\Lambda)\Bigg)^{(n+1)/(n+2)},
\end{equation}
where $c=\left(\frac{(n+1)^2\epsilon}{(n+2)C_4}\right)^{n+2}$.

The following estimate comes from (\ref{originalineq}), $$\int_{\Omega_{s,t_0}}\exp \left[\lambda\left( \frac{-\tilde{\varphi}-s}{A_{s,t_0}^{1/(n+2)}} \right)^{\frac{n+2}{n+1}}\right]\omega_{0}{}^n\mathrm{d}t\leq \int_{\Omega_{s,t_0}}\exp\Bigg\{\lambda c(-\psi+\Lambda) \Bigg\}\omega_{0}{}^n\mathrm{d}t.$$

%$$\gamma\left( \frac{-\varphi-s}{A_s^{1/(n+2)}} \right)^{\frac{n+2}{n+1}}\leq\gamma c_1(n)(-\psi+\Lambda),$$ where $\gamma$ is a universal constant to make $\gamma c=\alpha$ as in \ref{alpha}.

%Take exponential and integrate over $\Omega_s$, and suppose $A_{s,t_0,\delta}$ is uniformly bounded by $E$, then we have $$\int_{\Omega_s}\exp \left[\gamma\left( \frac{-\varphi-s}{A_{s}^{1/(n+2)}} \right)^{\frac{n+2}{n+1}}\right]\leq C_3(E),$$ we use $C_3(E)$ to emphasis the constant depends not only on the background data but also on the uniform bound $E$.

If the universal constant $\lambda$ is chosen to make $\lambda c=\alpha$, then by corollary \ref{alpha}, the right hand side of the above inequality is bounded by universal constants as well as a function of $A_{s,t_0}$. 

Let us suppose $E=\sup_{t_0\in[0,T-1)}\int_{M\times[t_0,t_0+1]}(-\tilde{\varphi})e^{nF}\omega_0{}^n\mathrm{d}t$, then we have 

\begin{equation}\label{theintineq}
    \int_{\Omega_{s,t_0}}\exp \left[\lambda\left( \frac{-\tilde{\varphi}-s}{A_{s,t_0}^{1/(n+2)}} \right)^{\frac{n+2}{n+1}}\right]\omega_{0}{}^n\mathrm{d}t\leq C(E),
\end{equation}
where $C(E)$ denotes a constant not only depends on the universal constants but also depends on $E$. We will use the same notation to denote a constant depending on the same parameters in the following of this paper. 

%It is important to get the implicit dependence on $\delta$, since later we will get all estimates independence on $\delta$.

To end this section, we will use the De Giorgi iteration method to derive the $C^0$ estimate by assuming $E$ is universally bounded. In the next section, we will apply the ABP estimate to get the universal bound on $E$ which will complete the proof of theorem \ref{MainThm}. To set up for the iteration procedure, we need such inequality to run the iteration
\begin{equation}\label{iteration}
    r\phi_{t_0}(s + r) \leq A_{s, t_0} \leq B_0\phi_{t_0}(s)^{1 + \delta_0},
\end{equation}
where $\phi_{t_0}(s)=\int_{\Omega_{s,t_0}}e^{nF}\omega_0{}^n\mathrm{d}t$. We will use $\phi$ to denote $\phi_{t_0}$ for simplicity.

% where in our application, let $\phi_{t_0,\delta}$ is given by $\phi_{t_0,\delta}(s)=\int_{\Omega_{s,t_0,\delta}}e^{nF}$. We will use $\phi(s)$ denote  $\phi_{t_0,\delta}(s)$.

The following lemma is the method of De Giorgi iteration mentioned above,

\begin{lemma}
    Let $\Phi:\mathbb{R}_{+}\rightarrow\mathbb{R}_{+}$ be a decreasing right continuous function with \newline $\lim_{s\rightarrow\infty}\Phi(s)=0$. Moreover, assume $r\Phi(s+r)\leq B_0\Phi(s)^{1+\delta_0}$ for some constant $B_0>0$ and all $s>0$ and $r\in[0,r]$. Then there exists constant $S_{\infty}=S_{\infty}(\delta_0,B_0,s_0)>0$ such that $\Phi(s)=0$ for all $s\geq S_{\infty}$, where $s_0$ is defined during the proof.
\end{lemma}

\begin{proof}
    Fix an $s_0>0$ such that $\Phi(s_0)^{\delta_0}<\frac{1}{2B_0}$. This $s_0$ exists since $\Phi(s)\rightarrow 0$ as $s\rightarrow \infty$. Define $\{ s_j \}$ by $s_{j+1}=\sup\{ s>s_j|\phi(s)>\frac{1}{2}\phi(s_j) \}$. Thus $$s_{j+1}-s_j\leq B_0\Phi(s_j)^{1+\delta}/\Phi(s_{j+1})\leq2B_0\Phi(s_j)^{\delta}\leq2B_02^{-j\delta_0}\Phi^{\delta_0}\leq2^{-j\delta_0}.$$

    Let $S_{\infty}=s_0+\sum_{j\geq0}(s_{j+1}-s_{j})\leq s_0+\frac{1}{1-2^{-\delta_0}}$ which completes the proof. 
\end{proof}

In our application, $\Phi$ are chosen to be $\phi_\delta$. Let us derive two sides of (\ref{iteration}) respectively. The left hand side of (\ref{iteration}) can be derived by definition, and the proof can be found in \cite{GPT}. To get the right hand side of (\ref{iteration}), we need to apply the following inequality coming from Young's inequality,

\begin{equation}\label{Young}
    \int_{\Omega_{s}}v^pe^{nF}\leq \|e^{nF} \|_{L^1(\log L)^p(M\times[t_0,t_0+1])}+C_p\int_{\Omega_s} e^{2v}.
\end{equation}

If we choose $v=\frac{\lambda}{2}\left(\frac{-\tilde{\varphi}-s}{A_s^{1/(n+2)}} \right)^{\frac{n+2}{n+1}}$, then by the above inequality (\ref{Young}) we will have the following inequality

\begin{equation}
    \int_{\Omega_s}(-\tilde{\varphi}-s)^{\frac{(n+2)p}{n+1}}e^{nF}\omega_{0}{}^n\mathrm{d}t\leq  C(E) A_s^{\frac{p}{n+1}}.
\end{equation}
% where $C(E)$ is a constant has the same dependence as in (\ref{theintineq}), but may not be the same one.

%gives $\int_{\Omega_s}(-\varphi-s)^{\frac{(n+2)p}{n+1}}e^{nF}\omega_{0}^n\leq A_s^{\frac{p}{n+1}}C_4(E)$. If we apply Holder inequality, we will have 

Thus the right hand side can be derived by the following estimate,

\begin{align*}
    A_s=&\int_{\Omega_s}(-\tilde{\varphi}-s)e^{nF}\omega_{0}{}^n\mathrm{d}t \leq\left(\int_{\Omega_s} (-\tilde{\varphi}-s)^{\frac{(n+2)p}{n+1} }e^{nF}\right)^{\frac{n+1}{(n+2)p}}\cdot\left( \int_{\Omega_s}e^{nF}\right)^{\frac{1}{q}}\\
    \leq& \Tilde{C}(E)A_s^{\frac{1}{n+2}}\phi^{1-\frac{n+1}{p(n+2)}},
\end{align*}
where the first line is by $\mathrm{H\ddot{o}lder}$'s inequality, and $p$ is the $\mathrm{H\ddot{o}lder}$ coefficient $\frac{n+1}{p(n+2)}+\frac{1}{q}=1$. In (\ref{iteration}), we can choose $\delta_0=1+\frac{p-n-1}{p(n+1)}$ and $B_0=\Tilde{C}(E)^{(n+2)/(n+1)}$.

%The above inequality gives us $$A_s\leq B_0(E)\phi(s)^{1+\delta},$$ with $\delta=\frac{p-(n+1)}{p(n+1)}>0$, since $p>n+1$ by assumption. And following similar argument, we have $r\phi(r+s)\leq A_s$. Combine the two inequalities, we could run De Giorgi iteration since we have the following lemma.

To complete this section, we need to get an explicit expression on $s_0$. By Chebyshev's inequality $$\phi(s)\leq\frac{1}{s}\int_{\Omega_s}(-\tilde{\varphi})e^{nF}\omega_{0}^n\leq \frac{E}{s},$$ we can let $s_0=(2B_0)^{1/\delta_0}E$ be the one we need.

In other word, we get the theorem from the above arguments,

\begin{thm}\label{entbyeng}
    Let $\tilde{\varphi}$ be as in \ref{MainThm}, then we have 
$$\sup_{M\times[0,T)}|\tilde{\varphi}|\leq C(E).$$ Moreover, if $E$ can be controlled by a universal constant, then we have theorem \ref{MainThm}.
\end{thm}

%Following proof in \cite{GPT}, $B_0$ and $\delta_0$ do not depend on $(t_0,\delta)$, which gives us a uniform bound on $C^0$ under the assumption $A_s$ is uniformly bounded. 

%xxxxxxxxxxxxxxxxxxxxxxxx details are needed xxxxxxxxxxxxxxxxxxxxxxxx

\section{Energy bounds by ABP estimate}\label{seceng}

In this section we will use a parabolic version of ABP estimate by Krylov-Tso in \cite{PABP1,PABP2}, to give us a uniform energy bound. This approach was introduced in \cite{CC3} which is an analogue to the elliptic version in \cite{GPT}.

Let $u$ be a function defined on $D=\Omega\times[0,T]$, where $\Omega$ is a bounded domain in $\mathbb{R}^n$, then the parabolic ABP estimate says that

\begin{equation}\label{ParaABP}
    \sup_D u\leq\sup_{\partial_PD}u+C_n(\mathrm{diam}\Omega)^{\frac{n}{n+1}}\left(\int_{\Gamma} |\partial_t u\det D^2_x u|\mathrm{d}x\mathrm{d}t\right)^{\frac{1}{n+1}},
\end{equation}
where $\partial_T D$ is the parabolic boundary of $D$, and $\Gamma=\{(z,t)| \partial_t u\geq 0, \det D^2_x u\leq 0 \}$.

As mentioned in section \ref{sectaux}, we want to construct a family of local auxiliary equations. The auxiliary equations in this section are chosen to be

\begin{equation}\label{aux2}
    \left\{
    \begin{aligned}
        &(-\partial_t\psi_{t_0})\omega_{\psi_{t_0}}{}^n=\frac{(|F|^p+1)\chi_{M\times[t_0,t_0+1]}}{\int_{M\times[t_0,t_0+1]}(|F|^p+1)e^{nF}\omega_{0}{}^n\mathrm{d}t}\omega_{0}{}^n\\
        &\psi_{t_0}(\cdot,0)=0.
    \end{aligned}
    \right.
\end{equation}

We will use $\psi$ to denote $\psi_{t_0}$ without specifying. We also need to use $\tau_k$ to smoothen the singularities, but we will still keep estimating based on the family of auxiliary equations (\ref{aux2}) in the following to simplify our computation.
%computations as well.

Parallel with lemma \ref{keylemma}, the following lemma plays a crucial role in this section.

\begin{lemma}\label{keylem2}
    Suppose $\varphi$ as in theorem \ref{MainThm}, and $\psi$ is a solution of (\ref{aux2}), then there exists a universal constant $C$, such that the following hold on $M\times[t_0,t_0+1]$,
    \begin{equation}\label{keyest2}
        -\epsilon(-\psi+\Lambda)^{\beta}-\tilde{\varphi}\leq C,
    \end{equation}
    where $\epsilon$ and $\beta$ will be defined during the proof.
\end{lemma}

Let $\rho$ be the test function defined by $\rho=-\epsilon(-\psi+\Lambda)^{\beta}-\tilde{\varphi}$, and the operator $L$ be the linearization as above. To prove the lemma \ref{keylem2}, we only need to restrict $\rho$ into its the positive part. To be more precise, let us consider $h_{s}(x)=x+\sqrt{x^2+s}$, and use $h_{s}(\rho)$ to approximate $\rho_{+}$. We define the constants $\epsilon$ and $\Lambda$ in the following way $$\beta\epsilon\Lambda^{\beta-1}=1/4,\ \Lambda=C_5\mathrm{Ent}_p(F)^{\frac{1}{(n+1)(1-\beta)}},$$ where $C_5$ is a universal constant to be derived during the proof.

Let us consider $h_s(\rho)^{b}$ where $b=1+\frac{1}{2n+2}$, and assume $h_s(\rho)^{b}$ attains its maximal value $M$ at some point $x_0\in M\times[0,T)$. Moreover, we can assume $M>1$, otherwise there is nothing to prove. Let us apply the parabolic ABP estimate for $H=h_s(\rho)^{b}\cdot\eta$, where $\eta$ is a cut-off function defined below.

%denote the maximal value of $h_s(\rho)^{b}$ by $M$. We want to the apply the parabolic ABP estimate for $H=h_s(\rho)^{b}\cdot\eta$, where $\eta$ is a cut off function defined below. Let $x_0$ be the maximum point of $h_s(\rho)^b$ over the space-time and we can further assume $M>1$ otherwise there is nothing to prove. 
Assume $r=\min\{1,\mathrm{inj}(M,\omega_{0})\}$ and the cut-off function $\eta$ is defined in the following way,
\begin{align}
    &\eta\equiv1\ \mathrm{on}\ B_{\omega_0}(x_0,r/2)\times[t_0,t_0+1], \label{eta1}\\
    &\eta\equiv1-\theta\ \mathrm{on}\ \{ M\backslash B_{\omega_0}(x_0,3r/4)\}\times[t_0,t_0+1],\label{eta2}\\
    &|\nabla \eta|_{\omega_0}^2\leq 10\theta^2/r^2,\label{eta3}\\
    &\ |\nabla^2 \eta|_{\omega_0}\leq 10\theta/r^2,\label{eta4}
\end{align}
where $\theta=\min\{\frac{r^2\beta\Lambda^{\beta-1}}{100M^{1/b}}, \frac{r^2}{100n}\}$. The advantage of choosing such cut-off function $\eta$ is that $\eta$ is a space function with vanishing time derivative.

%Before we start our proof of lemma \ref{keylem2}, I will give a brief explanation on why we need such assumptions on the cut off function $\eta$. Conditions \ref{eta1} and \ref{eta2} are different from the usual cut off functions with outside region set to be $0$. This was considered in \cite{CC3} to get the $L^{\infty}$ estimate for a system of equations. The technical consideration is we want to isolate the maximal point $x_0$ to make it a strict maximal point of our test function $H$. For condition \ref{eta3} and \ref{eta4}, the norms are chosen to be a fixed one. In flow problem one may suspect a control under moving norms. As in the famous local version of Shi's estimate for Ricci flows, the cut off function are argued by a more subtle way. Starting with similar conditions, he finally got a derivative bound under moving metrics along the Ricci flow. But as what we are going to show in the following, we do not need to get similar properties of $\eta$, since we can use a factor $\tr_{\omega_{\varphi}}\omega_{0}$ to measure the difference of the fixed norm and moving norms. It turns out this factor $\tr_{\omega_{\varphi}}\omega_{0}$ could be absorbed by subtle choice of the constants.

We are ready to prove the lemma \ref{keylem2}. The following inequality can be derived directly by applying the operator $L$ on $H$,
\begin{equation}\label{mainineq}
    LH\geq bh'h^{b-1}(-\partial_t \rho)\eta+(\lap_{\omega_{\varphi}}h^b)\eta+2\mathrm{Re}\left<\nabla h^b,\bar{\nabla}\eta\right>_{\omega_{\varphi}}+h^{b}\lap_{\omega_{\varphi}}\eta.
\end{equation}

We will consider each of the terms respectively to get good controls. Let us control the last two terms firstly, 
\begin{equation}\label{in1}
    h^{b}\lap_{\omega_{\varphi}}\eta\geq-\frac{10\theta}{r^2}h^b\tr_{\omega_{\varphi}}\omega_{0},
\end{equation}

\begin{equation}\label{in2}
    2\mathrm{Re}\left<\nabla h^b,\bar{\nabla}\eta\right>\geq-\frac{b(b-1)}{2}|\nabla h|^2_{\omega_{\varphi}}h^{b-2}-\frac{2b}{b-1}h^b|\nabla\eta|^2_{\omega_{\varphi}}.
\end{equation}

%It is also straightforward for us to compute the second term in the following way,
Then, expand the second term and get,
\begin{equation}\label{in3}
    (\lap_{\omega_{\varphi}}h^b)\eta=b(b-1)|\nabla h|^2_{\omega_{\varphi}}h^{b-2}\eta+bh'h^{b-1}\left(\lap_{\omega_{\varphi}}\rho\right)\eta+b|\nabla \rho|^2_{\omega_{\varphi}}h''h^{b-1}\eta
\end{equation}

Combine (\ref{mainineq}), (\ref{in1}), (\ref{in2}) and (\ref{in3}), and notice that the first term in (\ref{in3}) can absorb the first term in (\ref{in2}) and the third term of (\ref{in3}) is positive, thus we have 
\begin{equation}\label{LH1}
    \begin{aligned}
    LH\geq& bh'h^{b-1}(-\partial_t \rho)\eta+bh'h^{b-1}\left(\lap_{\omega_{\varphi}}\rho\right)\eta-\frac{2b}{b-1}h^b|\nabla\eta|^2_{\omega_{\varphi}}-\frac{10\theta}{r^2}h^b\tr_{\omega_{\varphi}}\omega_{0}\\
    \geq& bh'h^{b-1}(L\rho)\eta-\frac{2b}{b-1}\frac{100\theta^2}{r^2}h^b\tr_{\omega_{\varphi}}\omega_{0}-\frac{10\theta}{r^2}h^b\tr_{\omega_{\varphi}}\omega_{0}.
    \end{aligned}
\end{equation}

As we mentioned, the derivatives of the cut-off function $\eta$ will produce $\tr_{\omega_{\varphi}}\omega_{0}$ terms in (\ref{LH1}) which will be absorbed in the later estimates. The $L\rho$ term is the main term of (\ref{LH1}) and it has the same structure as the main term of the test function appearing in lemma \ref{keylemma}. This fact motives the following argument. %The main component of (\ref{LH1}), i.e. $L\rho$ behaves similar as the test function in the proof of lemma \ref{keylemma} which motivates the following estimates.

Let us compute $L\rho$ and drop the positive term $\beta(1-\beta)\epsilon(-\psi+\Lambda)^{\beta-2}|\nabla\psi|^2$, then we have 
\begin{equation*}
    \begin{aligned}
         L\rho\geq& -\beta\epsilon\dot{\psi}(-\psi+\Lambda)^{\beta-1}+\dot{\tilde{\varphi}}+\beta\epsilon(-\psi+\Lambda)^{\beta-1}\lap_{\omega_{\varphi}}\psi\\
         &+\beta(1-\beta)\epsilon(-\psi+\Lambda)^{\beta-2}|\nabla\psi|^2-\lap_{\omega_{\varphi}}\tilde{\varphi}\\
         \geq&-\beta\epsilon\dot{\psi}(-\psi+\Lambda)^{\beta-1}+\dot{\varphi}+\beta\epsilon(-\psi+\Lambda)^{\beta-1}\lap_{\omega_{\varphi}}\psi-\lap_{\omega_{\varphi}}\varphi - C_3.
    \end{aligned}
\end{equation*}

Since $\lap_{\omega_{\varphi}}\varphi+\tr_{\omega_{\varphi}}\omega_{0}=n$ and $\lap_{\omega_{\varphi}}\psi+\tr_{\omega_{\varphi}}\omega_{0}=\tr_{\omega_{\varphi}}\omega_{\psi}$, we have
\begin{equation}\label{Lrho1}
    L\rho\geq\beta\epsilon(-\psi+\Lambda)^{\beta-1}\left(-\dot{\psi}+\tr_{\omega_{\varphi}}\omega_{\psi} \right)+\dot{\varphi}+(1-\beta\epsilon\Lambda^{\beta-1})\tr_{\omega_{\varphi}}\omega_{0}-C_3-n.
\end{equation}

The $\tr_{\omega_{\varphi}}\omega_{0}$ term in (\ref{Lrho1}) will serve as a good term to absorb the last two terms in (\ref{LH1}). The estimate for the rest terms in (\ref{Lrho1}) follows the same idea in (\ref{geo}), (\ref{4keylemma}) and (\ref{middlesteps}).
\begin{equation}
    \begin{aligned}\label{untitled}
        &\beta\epsilon(-\psi+\Lambda)^{\beta-1}\left(-\dot{\psi}+\tr_{\omega_{\varphi}}\omega_{\psi} \right)+\dot{\varphi}-C_3-n\\
        \geq&(n+1)\beta\epsilon(-\psi+\Lambda)^{\beta-1}\Tilde{f}^{1/(n+1)}\exp\left(-\frac{1}{n+1}\dot{\varphi}\right)+\dot{\varphi}-C_3-n\\
        =&\left[(n+1)\beta\epsilon(-\psi+\Lambda)^{\beta-1}\Tilde{f}^{1/(n+1)}+(\dot{\varphi}-C_3-n)\exp\left(\frac{1}{n+1}\dot{\varphi}\right)\right]\exp\left(-\frac{1}{n+1}\dot{\varphi}\right)\\
        \geq&\left[(n+1)\beta\epsilon(-\psi+\Lambda)^{\beta-1}\Tilde{f}^{1/(n+1)}-C_4\right]\exp\left(-\frac{1}{n+1}\dot{\varphi}\right),
    \end{aligned}
\end{equation}
where $$\Tilde{f}=\frac{(|F|^p+1)\chi_{M\times[t_0,t_0+1]}}{\int_{M\times[t_0,t_0+1]}(|F|^p+1)e^{nF}\omega_{0}{}^n\mathrm{d}t}.$$ %coming from right hand side of the auxiliary equation (\ref{aux2}).

What we get right now is the following inequality, note that the cut-off function satisfies $\eta\geq9/10$ for any points in the space-time $M\times[0,T)$.
\begin{equation}\label{LH2}
    \begin{aligned}
    LH\geq&  \frac{9}{10}(n+1)bh'h^{b-1}\left[\beta\epsilon(-\psi+\Lambda)^{\beta-1}\Tilde{f}^{1/(n+1)}-\frac{C_4}{n+1}\right]\exp\left(-\frac{1}{n+1}\dot{\varphi}\right)\\
    &+bh^{b-1}\left[\frac{9}{10}bh'(1-\beta\epsilon\Lambda^{\beta-1})-\frac{200\theta^2}{(b-1)r^2}h-\frac{10\theta}{br^2}h\right]\tr_{\omega_{\varphi}}\omega_{0}.
    \end{aligned}
\end{equation}

The $\tr_{\omega_{\varphi}}\omega_{0}$ term is positive since the choice of $\epsilon, \beta$ and $\Lambda$ and the first term requires more subtle arguments after applying the parabolic ABP estimate. Since we only care about the positive component $\rho_{+}$, we should get the following estimate in the two domains $\Omega_+=\{\rho>0 \}$ and $\Omega_{-}=\{\rho\leq0\}$ separately.

On $\Omega_{-}$, we have $$0\leq h_s(\rho)=\rho+\sqrt{\rho^2+s}=\frac{s}{\sqrt{\rho^2+s}-\rho}\leq\sqrt{s},$$ and $$1\leq h_s'(\rho)=1+\frac{\rho}{\sqrt{\rho^2+s}}\geq0.$$

Combine the two bounds on $h$ and $h'$ and inequality (\ref{LH2}), we have $$LH\geq bh^{b-1}\left[-C-\frac{200\theta^2}{(b-1)r^2}h\tr_{\omega_{\varphi}}\omega_{0}-\frac{10\theta}{br^2}h\tr_{\omega_{\varphi}}\omega_{0}\right],$$ on $\Omega_{-}$.

% Roughly speaking, the right hand side is of $O(s^{(b-1)/2})$ which converges to $0$ as $s\rightarrow0$. Thus, this part is bounded even under the integration.

On the other hand, $1\leq h'\leq2$ on $\Omega_+$. By the choice the constants, the $\tr_{\omega_{\varphi}}\omega_{0}$ term in (\ref{LH2}) is positive. Therefore, we have $$LH\geq Cbh^{b-1}\left[\beta\epsilon(-\psi+\Lambda)^{\beta-1}\Tilde{f}^{1/(n+1)}-1\right]\exp\left(-\frac{1}{n+1}\dot{\varphi}\right)$$

Combining the above two cases, we obtain
\begin{equation}\label{preABP}
   \begin{aligned}
        LH\geq& bh^{b-1}\left[-C-\frac{200\theta^2}{(b-1)r^2}h\tr_{\omega_{\varphi}}\omega_{0}-\frac{10\theta}{br^2}h\tr_{\omega_{\varphi}}\omega_{0}\right]\chi_{\Omega_{-}}\\
        &+Cbh^{b-1}\left(\beta\epsilon(-\psi+\Lambda)^{\beta-1}\Tilde{f}^{1/(n+1)}-1\right)\exp\left(-\frac{1}{n+1}\dot{\varphi}\right)\chi_{\Omega_+}.
   \end{aligned}
\end{equation}

We need also discuss how to control the operator $L$ on the domain $\Gamma$. The estimate can be derived by $$LH=-\frac{\partial}{\partial t}H+\lap_{\omega_{\varphi}}H\leq-(n+1)\left(\left(-\frac{\partial}{\partial t}H\right)\cdot \det\mathrm{D}^2 H \cdot\frac{\omega_{0}{}^n}{\omega_{\varphi}{}^{n}} \right)^{\frac{1}{n+1}} ,$$ which connects our operator with the parabolic ABP estimate.

We are well prepared to apply the parabolic ABP estimate for (\ref{preABP}) based on above arguments. Denote $D=B_r(x_0)\times[t_0,t_0+1]$, then the parabolic ABP estimate (\ref{ParaABP}) tells us that
\begin{equation}\label{PostABP}
    \begin{aligned}
         &\sup_{D}(H)-\sup_{\partial_P D}(H)\leq\\
         &C_n\left( \int_{E\cap\Omega_+} h^{(n+1)(b-1)}\left(\beta\epsilon(-\psi+\Lambda)^{\beta-1}\Tilde{f}^{1/(n+1)}-1\right)^{n+1}_{-}\exp\left(-\dot{\varphi}\right)\frac{\omega_{\varphi}{}^n}{\omega_{0}{}^{n}} \omega_{0}{}^n\mathrm{d}t \right.\\
         &+\left.\int_{E\cap\Omega_{-}} h^{(n+1)(b-1)} \left(C+\frac{200\theta^2}{(b-1)r^2}h\tr_{\omega_{\varphi}}\omega_{0}+\frac{10\theta}{br^2}h\tr_{\omega_{\varphi}}\omega_{0}\right)^{n+1}\frac{\omega_{\varphi}{}^n}{\omega_{0}{}^{n}} \omega_{0}{}^n \mathrm{d}t \right)^{1/2n+1}\\
         &\leq C_n\left( \int_{E\cap\Omega_+} h^{1/2}\left(\beta\epsilon(-\psi+\Lambda)^{\beta-1}\Tilde{f}^{1/(n+1)}-1\right)^{n+1}_{-} e^{nF}\omega_{0}{}^n\mathrm{d}t + C_6s^{(n+1)(b-1)/2}\right)^{1/2n+1} \\
    \end{aligned}
\end{equation}

% Note that for later we will see, the $\omega_{\varphi}$ works as a determination of the linearization of $\log$ of the operator which show be understood as a function. By the cut off function, the left hand side of (\ref{PostABP}) is $\frac{1}{10}M$. Then let us control the second integration in (\ref{PostABP}). As $s\rightarrow0$, we have $\phi^{1/2}=O(s)$, and by use geometric-arithmetic inequality, we have universal bounds on $\left(C+\frac{200\theta^2}{(b-1)r^2}\phi\tr_{\omega_{\varphi}}\omega_{0}+\frac{10\theta}{br^2}\phi\tr_{\omega_{\varphi}}\omega_{0}\right)^{n+1}\omega_{0}^{-n}\omega_{\varphi}^n$. Thus the second integration is bounded from above by $As$, where $A$ is some universal constant.

% Before we control the first integration, we should simplify the exponential factor by $$\exp(-\dot{\varphi})\omega_{0}^{-n}\omega_{\varphi}^n=e^{nF}.$$ 

By $\beta\epsilon(-\psi+\Lambda)^{\beta-1}\Tilde{f}^{1/(n+1)}-1\leq0$, we have \begin{equation}\label{Fcontrol}
    |F|\leq\Psi^{1/p}(\beta\epsilon)^{-p/(n+1)}(-\psi+\Lambda)^{(1-\beta)(n+1)/p}=\frac{\alpha}{2}(-\psi+\Lambda)^{(1-\beta)(n+1)/p},
\end{equation}
where $\alpha$ is the constant in the lemma \ref{alpha}, the universal constant $C_5$ could be chosen to make the last equality in (\ref{Fcontrol}) hold. Moreover, $h(\rho) \leq 2\rho+\sqrt{\delta}$. Then by combining with the inequality (\ref{PostABP}) and letting $s\rightarrow0$, we have 
\begin{equation}
    \begin{aligned}
        \frac{1}{10}M&\leq\sup_{D}(H)-\sup_{\partial_P D}(H)\\
        &\leq C\left( \int_{E\cap\Omega_+} (-\tilde{\varphi}+\sqrt{\delta})^{1/2}\exp\left(\frac{\alpha}{2}(-\psi+\Lambda)^{(1-\beta)(n+1)/p}\right) \omega_{0}{}^n\mathrm{d}t + C_6s^{1/4}\right)^{1/2n+1} \\
        &\leq C\left( \int_{E\cap\Omega_+}\left( -\tilde{\varphi} + \exp\left(\alpha(-\psi+\Lambda)^{(1-\beta)(n+1)/p}\right)\right) \omega_{0}{}^n\mathrm{d}t + C_6s^{1/4}\right)^{1/2n+1}
    \end{aligned}
\end{equation}

From lemma \ref{gap}, we have the $L^1$-estimate of $\int -\tilde{\varphi}\omega_0{}^n\leq C_3$ at any time slice. The exponential term could be controlled by lemma \ref{alpha} instead. The choice of $\beta$ is given by this way: if $p\leq n+1$, then $(1-\beta)(n+1)=p$; and if $p=n+1$, then $(1-\beta)(n+1)=p(1-N^{-1})<1$. This completes the proof of lemma \ref{keylem2}.

Similar as in \cite{GPT}, we could get the following energy estimate,
\begin{thm}\label{energymacase}
    Suppose $\varphi$ to be the $C^2$ solution defined in theorem \ref{MainThm}, we have the energy estimate $$\int_{M\times[t_0,t_0+1]}(-\tilde{\varphi})^Ne^{nF}\omega_{0}{}^n\mathrm{d}t\leq C,$$ for $N=\frac{n+1}{n+1-p}$ if $p\in[1,n+1)$, and for $N>0$ if $p>n+1$, where the constant $C$ is a universal constant.
\end{thm}

The theorem \ref{energymacase} together with theorem \ref{entbyeng} implies the main theorem \ref{MainThm}.

%The entropy boundness was applied in the Young's inequality (\ref{Young}) which gives a control on $\|e^{nF} \|_{L^1(\log L)^p(M\times[t_0,t_0+\delta])}\leq\delta\mathrm{Ent}_p(F)$. Since when choosing pairs $(t_0,\delta)$, we can choose $\delta$ for every $t_0$ satisfying some conditions. This motives us to weaken the condition on $F$. 

\section{Some generalizations}\label{secgen}

In this section, we will derive some generalizations \ref{genmainthm} and \ref{genmainthm2} of our main theorem \ref{MainThm}.

The idea of theorem \ref{genmainthm} comes from the result of Chen-Cheng \cite{CC3} for general parabolic Hessian equations. The linearization of the flow (\ref{genflow}) is $Lu=-\partial_t u+G^{i\bar{j}}u_{i\bar{j}}$, where $G^{i\bar{j}}=\frac{1}{\mathcal{F}}\frac{\partial \mathcal{F}(\lambda[h_{\varphi}])}{\partial h_{i\bar{j}}}$. To prove theorem \ref{genmainthm}, we will use the family of auxiliary equations (\ref{aux}) and follow the same argument as in section \ref{secent} and \ref{seceng}.

The proof of main estimate (\ref{keylemma}) is tedious and we will only show the essential difference comparing with previous sections. When we apply the operator $L$ to the test function $-\epsilon(-\psi+\Lambda)^{\beta}-\tilde{\varphi}-s$, the Laplacian operator will be replaced by trace with respect to $G^{i\bar{j}}$. To be more precise, we have the following estimates,
\begin{align*}
    0\geq& L\left(-\epsilon\left(-\psi+\Lambda\right)^{\beta}-\tilde{\varphi}-s\right)\geq -\beta\epsilon(-\psi+\Lambda)^{\beta-1}\dot{\psi}+\dot{\tilde{\varphi}}\\
    &+\beta\epsilon(-\psi+\Lambda)^{\beta-1}\tr_{G}(\i\dd\psi)+\beta(1-\beta)\epsilon(-\varphi+\Lambda)^{\beta-2}|\partial\varphi|^2_{G}-\tr_{G}(\i\dd\varphi)\\
    \geq&\beta\epsilon(-\psi+\Lambda)^{\beta-1}\left(-\dot{\psi}\right)-\left(-\dot{\varphi}\right)+\beta\epsilon(-\psi+\Lambda)^{\beta-1}\tr_{G}(\i\dd\psi)-\tr_{G}(\i\dd\varphi)-C_3\\
    \geq&\beta\epsilon(-\psi+\Lambda)^{\beta-1}\left(-\dot{\psi}\right)-\left(-\dot{\varphi}\right)+\beta\epsilon(-\psi+\Lambda)^{\beta-1}\tr_{G}\omega_{\psi}-\tr_{G}\omega_{\varphi}-C_3\\
    \geq&\beta\epsilon(-\psi+\Lambda)^{\beta-1}\left(-\dot{\psi}+\tr_{G}\omega_{\psi}\right)-\left(-\dot{\varphi}+r+C_3\right).
\end{align*}

We also need to deal with the factor $-\dot{\psi}+\tr_{G}\omega_{\psi}$ as in inequalities (\ref{geo}), (\ref{4keylemma}) and (\ref{middlesteps}). The lower bound of the determinant on the condition of $\mathcal{F}$ will give us a lower bound on $\det G^{i\bar{j}}\geq\gamma\mathcal{F}^{-\frac{n}{r}}$. By comparing the two flows (\ref{aux}) and (\ref{genflow}) and the homogeneous degree $r$ condition, we have,
\begin{equation}\label{untitled2}
    \begin{aligned}
        -\dot{\psi}+\tr_{G}\omega_{\psi}&\geq(n+1)\sqrt[n+1]{\frac{fe^{nF}\omega_{0}{}^n}{\omega_{\psi}{}^n}\cdot\omega_{\psi}{}^n\det G^{i\bar{j}}}\\
        &\geq C_7 f^{1/(n+1)}\sqrt[n+1]{\left( \frac{e^{rF}}{\mathcal{F}}\right)^{n/r}}\\
        &\geq C_7 f^{1/(n+1)}\exp\left(-\frac{n}{r(n+1)}\dot{\varphi}  \right),\\
    \end{aligned}
\end{equation}
where $C_7$ is a universal constant.

Since the function $h(x)=(x-r-C_3)\exp\left(\frac{n}{r(n+1)}x\right)$ has a lower bound $c$ where $c=-\frac{r(n+1)}{n}\exp\left(\frac{nC_3-r}{r(n+1)}\right)$, we have the same estimate $$0\geq C_7\beta\epsilon(-\psi+\Lambda)^{\beta-1}f^{1/(n+1)}-\frac{r(n+1)}{n}\exp\left(\frac{nC_3-r}{r(n+1)}\right). $$ 

Then we can follow the same argument and the only thing we should be careful is the estimate (\ref{untitled}) which should be done by the trick in (\ref{untitled2}). The above argument gives the proof of theorem \ref{genmainthm}.

Instead of replacing the Hessian operator, we can also consider a much more general flow equation (\ref{star}). To get lemma \ref{keylemma} and \ref{keylem2}, we need lemma \ref{gap} to get the upper bound of the integral. Then we can apply the same procedure of estimate.

When $\Theta(y)=-1/y$, we have $$\int_M\dot{\varphi}=\int_M-e^{nF}\omega_0{}^n/\omega_{\varphi}{}^n<0.$$

When $\Theta(y)=y^a$ for $a>0$, we have $$\ddot{\varphi}=\frac{\mathrm{d}}{\mathrm{d}t}\left(\frac{\omega_{\varphi}{}^n}{e^{nF}\omega_0{}^n}\right)^a=a\lap_{\varphi}\dot{\varphi}\cdot\frac{\omega_{\varphi}{}^n}{e^{nF}\omega_0{}^n}\cdot\left(\frac{\omega_{\varphi}{}^n}{e^{nF}\omega_0{}^n}\right)^{a-1}=a\dot{\varphi}\lap_{\varphi}\dot{\varphi}.$$

Consider the first variation of the functional $\int_M\dot{\varphi}\omega_{\varphi}{}^n$, we have
\begin{equation*}
    \begin{aligned}
        &\frac{\mathrm{d}}{\mathrm{d}t}\int_M\dot{\varphi}\omega_{\varphi}{}^n=\int_M\ddot{\varphi}\omega_{\varphi}{}^n + \int_M\dot{\varphi}\frac{\mathrm{d}}{\mathrm{d}t}\left(\omega_{\varphi}{}^n\right)\\
        =&\int_M\ddot{\varphi}\omega_{\varphi}{}^n+\int_M\dot{\varphi}\lap_{\varphi}\dot{\varphi}\omega_{\varphi}{}^n\\
        =&(a+1)\int_M\dot{\varphi}\lap_{\varphi}\dot{\varphi}\omega_{\varphi}{}^n=-(a+1)\int_M|\nabla\dot{\varphi}|_{\omega_{\varphi}}^2\omega_{\varphi}{}^n\leq 0.
    \end{aligned}
\end{equation*}

Then the estimate of $\int_M\dot{\varphi}\omega_0{}^n$ follows from
\begin{equation*}
    \begin{aligned}
        &\int_M\dot{\varphi}\omega_0{}^n\leq\int_M\dot{\varphi}\omega_0{}^n-\int_M\dot{\varphi}\omega_{\varphi}{}^n+\int_M\dot{\varphi}(\cdot,0)\omega_{\varphi_0}{}^n\\
        \leq&\int_M\dot{\varphi}(\omega_0{}^n-\omega_{\varphi}{}^n)+\int_Me^{-anF}\left(\frac{\omega_{\varphi_0}{}^n}{\omega_0{}^n}\right)^a\omega_{\varphi_0}{}^n\\
        \leq&\int_M\dot{\varphi}(1-e^{nF}\dot{\varphi}^{1/a})\omega_0{}^n+C\int_Me^{-anF}\omega_0{}^n
    \end{aligned}
\end{equation*}
where $C$ is universal.

Let us consider a function $A(y)=y-ly^{1+\frac{1}{a}}$ defined on $y\in[0,\infty)$, where $l$ is some positive number. Then by calculus, we have $$A(y)\leq A\left(\frac{a^a}{l(a+1)^a}\right)\leq\frac{a^a}{(a+1)^{a+1}}l^{-a}.$$

Thus we have the following estimate, $$\int_M\dot{\varphi}\omega_0{}^n\leq \frac{a^a}{(a+1)^{a+1}}\int_Me^{-anF}\omega_0{}^n +C\int_Me^{-anF}\omega_0{}^n\leq CK,$$
which completes the proof of theorem \ref{genmainthm2}.

\bibliographystyle{plain}
\bibliography{ref.bib}

\end{document}